\documentclass[12pt]{amsart}
\usepackage[utf8]{inputenc}

\usepackage{lmodern,amsmath,amssymb,mathtools,microtype}
\usepackage[a4paper,margin=23mm]{geometry}
\usepackage{booktabs,array,xurl}
\usepackage[hidelinks]{hyperref}
\hypersetup{pdftitle={On the (n-2)-linear restriction estimate in n dimensions},pdfauthor={Faruk Temur}}

\newtheorem{theorem}{Theorem}[section]
\newtheorem{proposition}[theorem]{Proposition}
\newtheorem{lemma}[theorem]{Lemma}
\newtheorem{corollary}[theorem]{Corollary}
\theoremstyle{definition}
\newtheorem{external}{External input}
\theoremstyle{remark}
\newtheorem{remark}[theorem]{Remark}
\numberwithin{equation}{section}
\newcommand{\R}{\mathbb R}
\newcommand{\supp}{\operatorname{supp}}
\newcommand{\dist}{\operatorname{dist}}
\newcommand{\spn}{\operatorname{span}}
\newcommand{\ang}{\operatorname{ang}}
\newcommand{\gap}{\operatorname{gap}}
\newcommand{\Rap}{\operatorname{RapDec}}
\newcommand{\dd}{\,d}
\newcommand{\one}{\mathbf1}
\newcommand{\Nbh}{\mathcal N}
\newcommand{\T}{\mathcal T}
\newcommand{\pc}{p_{\mathrm c}}
\newcommand{\qg}{q_0}
\newcommand{\Pf}{\mathcal P(f)}
\allowdisplaybreaks[2]
\title[The $(n-2)$-linear restriction estimate]{The $(n-2)$-linear restriction estimate in $\mathbb R^n$}
\author{Faruk Temur}
\date{October 10, 2026.}
\begin{document}
\begin{abstract}
We present a proof  for the $(n-2)$-linear extension estimate  for curved surfaces in $\R^n$, $n\ge5$, at the local product exponent $2(n-1)/(n-2)^2$, with an arbitrarily small power loss. The hypotheses are transversality of the normals and nondegeneracy of the compressed shape operators on their common two-dimensional orthogonal complement. The argument combines polynomial partitioning, a descent through algebraic varieties, localized multilinear restriction, and positive packet-energy estimates. 
\end{abstract}
\maketitle

\section{Introduction}

The Fourier restriction problem asks when the Fourier transform of a function can be restricted to a set of Lebesgue measure zero. In its dual form it asks for integrability estimates for oscillatory integrals over that set. The Tomas--Stein theorem is the basic example of an estimate in which curvature improves integrability; see \cite{Tomas,Stein}.

The bilinear theory separates frequency pieces and exploits the interaction of their extensions. Wolff's theorem for the cone \cite{Wolff} and Tao's theorem for the paraboloid \cite{Tao} established sharp ranges apart from endpoints in their respective settings. Bennett, Carbery, and Tao \cite{BCT} developed the corresponding multilinear restriction and Kakeya theory. For $k$ transverse hypersurfaces in $\R^n$, their local product exponent is $2/(k-1)$, with an arbitrarily small power loss in the radius.

Under suitable curvature hypotheses the expected exponent is smaller:
\begin{equation}\label{generalexponent}
 p(n,k)=\frac{2(n+k)}{k(n+k-2)}.
\end{equation}
Bejenaru proved sharp estimates apart from endpoints for certain classes of curved hypersurfaces \cite{BejMulti}, and subsequently treated $k=n-1$ in a general geometric setting described by the shape operator \cite{BejCurv}.  The present draft concerns the next case $k=n-2$, and uses ideas and tools from \cite{BejCurv}.

Throughout the main argument,
\begin{equation}\label{dimensions}
 n\ge5,\qquad k=n-2,\qquad
 \pc=\frac{2(k+1)}{k^2}=\frac{2(n-1)}{(n-2)^2},\qquad
 \qg=\frac2{k-1}.
\end{equation}
In particular, $2/k<\pc<\qg\le1$. The exponents refer to the product of the $k$ extensions, not its $k$th root. Some values are
\[
\begin{array}{ccccc}
\toprule
n&5&6&7&8\\
\midrule
k&3&4&5&6\\
\pc&8/9&5/8&12/25&7/18\\
\qg&1&2/3&1/2&2/5\\
\bottomrule
\end{array}
\]
The restriction $n\ge5$ keeps all exponents used in the partition argument at most one. The value of \eqref{generalexponent} at $(n,k)=(4,2)$ is $3/2$, the corresponding bilinear critical exponent.

We use the algebraic tools of Guth \cite{Guth}, the localized restriction theorem of Bejenaru \cite{BejLoc} in the formulation of Beltran, Duncan, and Hickman \cite{BDH}, and the regularized Kakeya--Brascamp--Lieb theorem of Maldague \cite{Maldague}. The relevant statements are given in Section~\ref{externalinputs}. The article is self-contained relative to these external theorems; their deep proofs are not reproduced.

Let $S_1,\ldots,S_k$ be smooth compact hypersurface patches in $\R^n$, each lying in the interior of a fixed smooth extension. Put
\begin{equation}\label{extension}
 E_if(x)=\int_{S_i}e^{ix\cdot\zeta}f(\zeta)\dd\sigma_i(\zeta),
 \qquad i=1,\ldots,k.
\end{equation}
Write $N_i$ for a smooth unit normal and $A_i=-dN_i$ for the shape operator. We identify tangent spaces with linear subspaces of $\R^n$.

We assume, throughout the fixed extended patches,
\begin{equation}\label{trans}
 |N_1(\zeta_1)\wedge\cdots\wedge N_k(\zeta_k)|\ge\nu>0
\end{equation}
for every choice of frequency points. Define
\begin{equation}\label{Wdef}
 W=W(\zeta_1,\ldots,\zeta_k)
 =\spn\{N_1(\zeta_1),\ldots,N_k(\zeta_k)\}^{\perp}.
\end{equation}
Then $\dim W=2$ and $W\subset T_{\zeta_i}S_i$. Our curvature hypothesis is
\begin{equation}\label{curv}
 |\pi_W A_i(\zeta_i)v|\ge\kappa|v|
 \qquad(v\in W,\ i=1,\ldots,k),
\end{equation}
with $\kappa>0$ fixed. Thus the compression of each shape operator to $W$ is invertible, uniformly in the entire frequency tuple.

\begin{theorem}\label{main}
Under \eqref{trans} and \eqref{curv}, we have
\begin{equation}\label{localmain}
 \left\|\prod_{i=1}^kE_if_i\right\|_{L^{\pc}(B_R)}
 \le C_\varepsilon R^\varepsilon\prod_{i=1}^k\|f_i\|_{L^2(S_i)},
 \qquad R\ge2,
\end{equation}
for every $\varepsilon>0$, every ball $B_R\subset\R^n$, and arbitrary $L^2$ inputs. Together with epsilon removal this gives
\begin{equation}\label{globalmain}
 \left\|\prod_{i=1}^kE_if_i\right\|_{L^p(\R^n)}
 \le C_p\prod_{i=1}^k\|f_i\|_{L^2(S_i)},
 \qquad p>\pc.
\end{equation}
The constants may depend on $n$ and the fixed geometry.
\end{theorem}

Strictly elliptic patches satisfy \eqref{curv}: after a choice of orientation, $\langle A_iv,v\rangle\ge c|v|^2$, and for $v\in W$,
\[
 |\pi_WA_iv|\,|v|\ge|\langle A_iv,v\rangle|\ge c|v|^2.
\]
It is also satisfied by transversal patches of a light cone.

Nonvanishing Gaussian curvature alone does not imply \eqref{curv}. We extend the original inputs by zero to the fixed larger patches and use the larger extension operators when forming packets. This leaves room for the small frequency enlargements in the induction.

\section{Notation and external results}\label{externalinputs}

Write
\[
 \Pf=\prod_{i=1}^k\|f_i\|_2,\qquad
 \Nbh_b(Y)=\{x:\dist(x,Y)<b\}.
\]
The notation $CQ$ means the concentric dilation of a cube $Q$. Constants in $\lesssim$ are independent of the large radius $R$, but may depend on the fixed dimension and geometry. For a unit vector $v$ and a linear subspace $V$, put
\[
 \ang(v,V)=|\pi_{V^\perp}v|.
\]
For subspaces of equal dimension let $\gap(V,V')=\|\pi_V-\pi_{V'}\|$. We use $N_\zeta Y=(T_\zeta Y)^\perp$ for the normal space of a smooth submanifold. Wedges of subspaces are computed using orthonormal bases.

For all $q>0$, $\|F\|_q=(\int|F|^q)^{1/q}$. If $q\le1$, then
\begin{equation}\label{subadd}
 \left|\sum_jF_j\right|^q\le\sum_j|F_j|^q.
\end{equation}
All interpolation below is the integral form of H\"older's inequality; no Banach-space interpolation theorem is needed.

\begin{external}[Localized multilinear restriction (Bejenaru \cite{BejLoc})]\label{localinput}
Let $S'_i\subset S_i$ have codimension $c_i$ inside $S_i$, where $(c_1,\ldots,c_k)$ is a permutation of $(2,0,\ldots,0)$. Suppose
\[
 |N_{\zeta_1}S'_1\wedge\cdots\wedge N_{\zeta_k}S'_k|\ge c>0.
\]
On sufficiently small fixed charts, if the localized factor is supported within distance $\mu$ of $S'_i$, then
\begin{equation}\label{localizedinputeq}
 \left\|\prod_{i=1}^kE_ih_i\right\|_{L^{\qg}(B_L)}
 \le C_\eta L^\eta\mu\prod_{i=1}^k\|h_i\|_2,
 \qquad L\ge1,\quad 0<\mu<1,
\end{equation}
for every $\eta>0$.
\end{external}

 The normal-space dimensions are $3,1,\ldots,1$, totaling $k+2=n$. In the integral of the $\qg$ power the gain is $\mu^{\qg}$; in the quasi-norm it is $\mu$. The loss parameter in the cited integral formulation has been renamed.

We will use the uniform version for a smooth compact family of charts with common regularity bounds, a common chart radius, and a positive lower bound on the normal-space wedge. The quantitative justification for this use is given after Lemma~\ref{curvaturelemma}. Uniformity is not inferred merely from pointwise finiteness of the constants in \eqref{localizedinputeq}.

\begin{external}[Weighted multilinear Kakeya  (Maldague  \cite{Maldague})]\label{kakeyainput}
Fix independent unit vectors $n_1,\ldots,n_k\in\R^n$. For sufficiently small fixed directional neighborhoods, families $\T_i$ of infinite $b$-tubes in those directions and weights $c_{i,T}\ge0$ satisfy
\begin{equation}\label{kakeyainputeq}
 \int_{B_R}\prod_{i=1}^k
 \left(\sum_{T\in\T_i}c_{i,T}\one_T\right)^{1/(k-1)}
 \le C_\eta R^\eta b^{n-\eta}
 \prod_{i=1}^k\left(\sum_Tc_{i,T}\right)^{1/(k-1)},
 \qquad 1\le b<R.
\end{equation}
An infinite $b$-tube is the $b$-neighborhood of an affine line.
\end{external}

 Its regularized Brascamp--Lieb condition follows directly from independence. With $L_i=\pi_{n_i^\perp}$, for every subspace $U$,
\begin{equation}\label{BLcheck}
 \sum_{i=1}^k\dim L_iU
 =k\dim U-\#\{i:n_i\in U\}\ge(k-1)\dim U.
\end{equation}
These are the dimension inequalities for weights $1/(k-1)$.

A smooth complete intersection of dimension $s$ is a closed set
\[
 Z=Z(P_1,\ldots,P_{n-s})
\]
on which the gradients of the defining polynomials are linearly independent. A degree bound is a bound on their degrees. We allow $Z=\R^n$ when $s=n$.

\begin{external}[Algebraic tools]\label{alginput}
The following statements hold.
\begin{enumerate}
\item A finite nonzero absolutely continuous measure on $\R^s$ admits a polynomial partition of degree $O(D)$ into $\asymp D^s$ sign regions of comparable mass. The factors admit arbitrarily small generic constant perturbations preserving this comparison.
\item A degree-bounded smooth complete intersection $Z$ can be cut by finitely many smooth complete intersections $\Sigma_a\subset Z$ of dimension $s-1$ so that its tangent planes vary by less than any prescribed $\gamma>0$ on each component of $Z\setminus\bigcup_a\Sigma_a$. The number and degrees of the cuts depend only on the old degree, $n$, and $\gamma$.
\item If $\ell$ is a line and $Z$ has degree bound $L$, then
\[
 \{z\in Z\cap\Nbh_b(\ell):\ang(\ell,T_zZ)>\beta\}
\]
is covered by $C_{n,L}$ balls of radius $C_{n,L}b/\beta$.
\item A zero-dimensional smooth complete intersection has at most the product of its defining degrees many points.
\end{enumerate}
\end{external}

These are \cite[Theorem~5.5, Lemma~5.6 and the following paragraph, Lemma~5.7, and Theorem~5.8]{Guth}. The source uses the usual angle; our sine convention is comparable in the range used. A sign region need not be connected. A line not contained in the zero set enters at most $O(D)$ sign regions. At $s=n$ the tangent plane is constant and no cuts are needed.

\begin{external}[Epsilon removal]\label{epsiloninput}
Suppose the compact smooth patches are transverse and, for some $a>0$,
\[
 |E_i\psi(x)|\le C_\psi(1+|x|)^{-a},\qquad \psi\in C_c^\infty(S_i).
\]
If a local product estimate at exponent $u\ge2/k$ holds with every loss $R^\delta$, $\delta>0$, then the global product estimate holds at every exponent $v>u$.
\end{external}

This follows from \cite[Theorem~1.1]{BejLoc} by taking its loss parameter sufficiently small for each prescribed $v>u$.

\section{Wave packets and their energies}\label{packetssection}

We first subdivide the frequency patches into finitely many small fixed pieces. Expanding the product and using \eqref{subadd} shows that this does not restrict the original data. We may assume that the normals in the $i$th extended patch lie in a sufficiently small fixed neighborhood of an independent vector $n_i$.

\begin{lemma}[Local Plancherel estimate]\label{plancherel}
For a small smooth hypersurface patch and $L\ge1$,
\begin{equation}\label{planchereleq}
 \|Eh\|_{L^2(B(x_0,L))}^2\lesssim L\|h\|_2^2,
 \qquad \|Eh\|_\infty\lesssim\|h\|_2.
\end{equation}
\end{lemma}
\begin{proof}
Rotate the patch into a graph $(\xi,\phi(\xi))$, $\xi\in\R^{n-1}$, and absorb the smooth surface Jacobian into the input. At each fixed $x_n$, Plancherel in $x'\in\R^{n-1}$ gives
\[
 \int_{\R^{n-1}}|Eh(x',x_n)|^2\dd x'\lesssim\|h\|_2^2.
\]
The $x_n$-projection of the ball has length $2L$. The second assertion is Cauchy--Schwarz on the compact frequency patch.
\end{proof}

Fix $0<\tau<1/20$ and set
\begin{equation}\label{scales}
 w=R^{1/2+\tau/10},\qquad
 r=R^{1/2+\tau},\qquad
 \rho=R^{1-\tau}.
\end{equation}
For sufficiently large $R$, $R^{1/2}\ll w\ll r\ll\rho\ll R$. All constants multiplying these scales are fixed independently of $R$.

\begin{lemma}[Packet decomposition]\label{packetlemma}
Each input has an $L^2$ decomposition $f_i=\sum_{T\in\T_i}f_{i,T}$. The frequency support of an atom has diameter $CR^{-1/2}$, and its tube has radius $Cw$, length $CR$, and direction equal to the normal at the cap centre. If $e_{i,T}=\|f_{i,T}\|_2^2$, then
\begin{align}
 \sum_Te_{i,T}&\lesssim\|f_i\|_2^2,\label{energytotal}\\
 \left\|\sum_{T\in\mathcal A}f_{i,T}\right\|_2^2
 &\lesssim\sum_{T\in\mathcal A}e_{i,T}
 \qquad(\mathcal A\subset\T_i).\label{subsetenergy}
\end{align}
Uniformly on $B_R$, the sum of any subcollection of packets whose tubes do not contain $x$ is bounded by $C_NR^{-N}\|f_i\|_2$ for every $N$.
\end{lemma}
\begin{proof}
Choose a bounded-overlap smooth partition into frequency cubes of side $R^{-1/2}$. Expand each piece in an orthonormal Fourier series on a larger cube. Multiply the basis elements by a smooth cutoff equal to one on the cap piece and supported strictly inside the Fourier-series cube. The series reconstructs that piece exactly and introduces no periodic replicas. Parseval and bounded overlap prove \eqref{energytotal} and \eqref{subsetenergy}; each atom has norm comparable to the modulus of its coefficient.

For cap centre $\xi_\theta$ and Fourier frequency $v$, the phase is
\[
 x'\cdot\xi+x_n\phi(\xi)+v\cdot\xi.
\]
The associated line is $x'+x_n\nabla\phi(\xi_\theta)+v=0$, with normal direction $(-\nabla\phi(\xi_\theta),1)$. Rescale $\xi=\xi_\theta+R^{-1/2}\eta$. Derivatives of the rescaled phase of order at least two are bounded on $B_R$. Integration by parts gives
\[
 |Ef_{\theta,v}(x)|
 \lesssim_M R^{-(n-1)/4}|c_{\theta,v}|
 \left(1+\frac{|x'+x_n\nabla\phi(\xi_\theta)+v|}{R^{1/2}}\right)^{-M}.
\]
There are $O(R^{(n-1)/2})$ caps, while $v$ runs over a lattice with spacing comparable to $R^{1/2}$. Cauchy--Schwarz in the coefficients and summation of the lattice tails outside radius $w$ give the asserted bound, with $M$ chosen sufficiently large in terms of $N,\tau$. This argument uses absolute values and is uniform over subcollections. The tubes may be truncated at length $CR$ since only $B_R$ is considered.
\end{proof}

Below $\Rap(R)$ denotes a quantity bounded by $C_NR^{-N}$ for every fixed $N$. All packet selections use the original energies. Infinite series may first be truncated and then recovered by \eqref{subsetenergy}. The $L^\infty$ bound and \eqref{subadd} absorb packet errors in products and their positive powers. All spatial grids used below have polynomial cardinality in $R$, so summing these errors preserves rapid decrease.

\begin{lemma}[Summation of packet energies]\label{energysum}
Let $b=\lambda r$ with fixed $\lambda\ge1$, and let $Q$ run over a disjoint $b$-grid meeting $B_R$. For a sufficiently large fixed dilation constant, put
\begin{equation}\label{aQ}
 a_i(Q)^2=\sum_{T:T\cap CQ\ne\varnothing}e_{i,T}.
\end{equation}
Then for every $\eta>0$,
\begin{equation}\label{energysumeq}
 \sum_Q\prod_{i=1}^k a_i(Q)^{\qg}
 \lesssim_\eta (R/b)^\eta\Pf^{\qg}.
\end{equation}
\end{lemma}
\begin{proof}
Enlarge each packet axis to an infinite tube $\widetilde T$ of radius $C'b$. If $C'$ is sufficiently large, $T\cap CQ\ne\varnothing$ implies $Q\subset\widetilde T$. Thus
\[
 b^n\sum_Q\prod_i a_i(Q)^{\qg}
 \le\int_{B_{CR}}\prod_i
 \left(\sum_Te_{i,T}\one_{\widetilde T}(x)\right)^{1/(k-1)}\dd x.
\]
Apply External input~\ref{kakeyainput}, use \eqref{energytotal}, and divide by $b^n$. The constants in the dilations affect only the implicit constant.
\end{proof}

\section{Curvature and thin support}\label{curvaturesection}

\begin{lemma}\label{curvaturelemma}
After a sufficiently small fixed subdivision, let $H$ be a two-plane sufficiently close to $W$ for every frequency tuple under consideration. On uniform charts in a fixed extension,
\[
 S'_{i,H}=\{\zeta:\pi_HN_i(\zeta)=0\}
\]
is a smooth $(n-3)$-dimensional submanifold wherever it is nonempty. For sufficiently small $\mu$,
\begin{equation}\label{nearlevel}
 |\pi_HN_i(\zeta)|\le\mu
 \quad\Longrightarrow\quad
 \dist(\zeta,S'_{i,H})\lesssim\mu.
\end{equation}
Its normal space is quantitatively independent of the other $k-1$ surface normals.
\end{lemma}
\begin{proof}
For $\Phi(\zeta)=\pi_HN_i(\zeta)$,
\[
 D\Phi_\zeta v=-\pi_HA_i(\zeta)v.
\]
On $W$, condition \eqref{curv} makes $\pi_WA_i|_W$ invertible with smallest singular value at least $\kappa$. Replacing the output projection by $\pi_H$ is a small perturbation. Thus $D\Phi|_W$ has a uniformly bounded right inverse. The quantitative implicit function theorem gives a codimension-two level set and \eqref{nearlevel}. The fixed extension leaves room for the displacement.

At a zero of $\Phi$, $H\subset T_\zeta S_i$. Symmetry of the shape operator gives
\[
 T_\zeta S'_{i,H}
 =\{v\in T_\zeta S_i:\langle A_iv,h\rangle=0\text{ for all }h\in H\},
\]
and hence
\begin{equation}\label{normalformula}
 N_\zeta S'_{i,H}=\spn\{N_i,A_iH\}.
\end{equation}
The map $\pi_WA_i|_H:H\to W$ is invertible by \eqref{curv} and the closeness of $H$ and $W$. The $k$ normals, followed by the images under $A_i$ of an orthonormal basis of $H$, therefore have $n$-dimensional volume bounded below. Their lengths are bounded above. This proves the required normal-space transversality.

We record uniformity explicitly. Fix a compact set of admissible planes $H$ inside the preceding small neighborhood. The right inverses have a common bound, and the fixed surfaces have uniform derivative bounds. The parametric quantitative implicit function theorem supplies a common chart radius and bounds for the graph derivatives. A finite covering of the compact frequency patches and the plane parameter set yields uniformly finitely many such charts. The exterior-product bound above also has a common positive lower bound. These are the geometric bounds needed for the uniform localized estimate below.
\end{proof}

\begin{remark}[Uniformity of localization]\label{uniformremark}
The family of zero sets in Lemma~\ref{curvaturelemma} is parametrized on finitely many fixed-size charts by a compact set of two-planes and frequency centers. Local orthonormal frames on the Grassmannian suffice; no global frame is required. The quantitative implicit function theorem bounds all derivatives of these charts in terms of derivatives of the fixed surfaces and the right-inverse bound. The ambient normal wedges have a common positive lower bound.

To pass from these facts to a uniform constant in External input~\ref{localinput}, use the proof of \cite[Theorem~1.4, Sections~3--4]{BDH}. Its tubular coordinate changes and their Jacobians are uniform on this compact family after a common shrinking of the charts. The linear data, both for the localized ensemble and its unlocalized subensemble, range over a compact set of transverse data. The perturbative neighborhoods in \cite[Theorem~2.6]{BDH} cover this set. A finite subcover gives a common directional neighborhood and, for each fixed loss, a common Kakeya constant. In the recursion of \cite[Lemma~4.2]{BDH}, the wave-packet bounds, coordinate changes, and bounded-scale estimates therefore have common constants. The same recursion gives \eqref{localizedinputeq} uniformly for the entire family. This uses the stability in the cited proof, in addition to compactness of the family.
\end{remark}

\begin{corollary}\label{curvedlocal}
If one factor satisfies
\[
 \supp h_i\subset\{\zeta:|\pi_HN_i(\zeta)|\le C\mu\},
\]
then
\begin{equation}\label{curvedlocaleq}
 \left\|\prod_{i=1}^kE_ih_i\right\|_{L^{\qg}(B_b)}
 \lesssim_\eta b^\eta\mu\prod_{i=1}^k\|h_i\|_2.
\end{equation}
The constant is independent of the admissible plane $H$ and of $\mu$.
\end{corollary}
\begin{proof}
Use \eqref{nearlevel}, \eqref{normalformula}, and External input~\ref{localinput}. A fixed smooth frequency partition reduces to its small charts. Summing finitely many products by \eqref{subadd} costs a fixed constant.
\end{proof}

\section{Projection and cell counting}\label{projectionsection}

Two elementary consequences of the algebraic tools will be used at every stage.

\begin{lemma}[Lifting a short segment]\label{lifting}
Let $Z$ be a smooth complete intersection, $\Sigma\subset Z$ a closed union of cuts, and $Z_0$ a component of $Z\setminus\Sigma$. Suppose
\[
 \gap(T_zZ,V)<\gamma_0<1/10\qquad(z\in Z_0),
 \qquad \dim V=\dim Z.
\]
Let $P=P_V\circ\pi_V$. If
\[
 |x-z|<b,\quad z\in Z_0,\quad
 \dist(z,\Sigma)>10b,\quad
 \dist(x,Z(P))<b,
\]
then
\begin{equation}\label{walllift}
 \dist(x,Z\cap Z(P))<5b.
\end{equation}
\end{lemma}
\begin{proof}
Since $Z(P)=\pi_V^{-1}(Z(P_V))$, choose $y\in Z(P_V)$ with
$|y-\pi_Vx|<b$. Then $|y-\pi_Vz|<2b$. The map $\pi_V|_{T_uZ}$ is invertible on $Z_0$, with inverse norm at most two. Solve
\[
 u'(t)=\bigl(\pi_V|_{T_{u(t)}Z}\bigr)^{-1}(y-\pi_Vz),
 \qquad u(0)=z.
\]
Its speed is less than $4b$, so a solution up to time one stays in $\overline B(z,4b)$, away from $\Sigma$. A maximal solution cannot end there: its bounded speed gives a limit point, and closedness and smoothness of $Z$ place that point in the same component, where the vector field extends. Thus the solution exists to time one.

Its projection is $\pi_Vu(t)=\pi_Vz+t(y-\pi_Vz)$. Therefore $u(1)\in Z\cap Z(P)$ and $|x-u(1)|<5b$. No global injectivity of $\pi_V|_Z$ is needed.
\end{proof}

\begin{lemma}[Trimmed cells]\label{cellcount}
Suppose $P=P_V\circ\pi_V$ has degree $O(D)$. Remove $\Nbh_b(Z(P))$ from its sign regions. Let $\Omega$ have the property that each $x\in\Omega$ has a witness $z\in Z$ with
\[
 |x-z|<b,\qquad \gap(T_zZ,V)<2\gamma.
\]
A packet tube of radius $Cw\ll b$ meets at most $O(D)$ trimmed sign regions. If
$\ang(v(T),V)>4\gamma$, then
\begin{equation}\label{farcount}
 \#\{O:T\cap\Omega\cap(O\setminus\Nbh_b(Z(P)))\ne\varnothing\}
 \le C_{L,\gamma},
\end{equation}
where $L$ bounds the degree of $Z$. This constant is independent of the new degree $D$.
\end{lemma}
\begin{proof}
If $x\in T$ lies outside the wall, its short segment to the tube axis avoids $Z(P)$. The endpoint on the axis has the same sign pattern. The restriction of $P$ to the line has $O(D)$ zeros, unless it is identically zero. In the latter case the whole tube is in the wall.

For each cell counted in \eqref{farcount}, choose $x_j$ and a witness $z_j$. Then
\[
 \dist(z_j,\ell)\lesssim b,\qquad
 \ang(v(T),T_{z_j}Z)>2\gamma,
\]
where $\ell$ is the tube axis. External input~\ref{alginput}(3) covers these witnesses, and hence the $x_j$, by boundedly many balls of radius $C_{L,\gamma}b$.

For different sign regions, a factor has opposite signs at $x_j,x_k$. The segment between them meets $Z(P)$ at some $y$. Since both endpoints are at distance at least $b$ from $Z(P)$,
\[
 |x_j-x_k|=|x_j-y|+|y-x_k|\ge2b.
\]
Packing separated points in the covering balls proves \eqref{farcount}.
\end{proof}

\section{The dimension descent}\label{descentsection}

Fix
\begin{equation}\label{pchoice}
 \pc<p<\qg.
\end{equation}
All packets in the descent are formed once, at radius $R$. Every selected input has the form
$g_i=\sum_{T\in\mathcal A_i}f_{i,T}$, with $\mathcal A_i\subset\T_i$.

Let
\[
 M_0=\int_{B_R}\prod_{i=1}^k|E_if_i|^p.
\]
A stage consists of selected functions $g_i$, a smooth complete intersection $Z$ of dimension $s$, fixed constants $\lambda,K\ge1$, and a measurable set
\[
 \Omega\subset B_R\cap\Nbh_b(Z),\qquad b=\lambda r,
\]
such that
\begin{equation}\label{stageinvariant}
 M_0\le K M+\Rap(R)\Pf^p,\qquad
 M=\int_\Omega\prod_{i=1}^k|E_ig_i|^p.
\end{equation}
The degree of $Z$, $\lambda$, and $K$ are bounded independently of $R$. They may increase when the dimension decreases. Whenever a restriction retains a fraction $c$ of the current mass, we increase $K$ by the factor $c^{-1}$. Initially $Z=\R^n$, $s=n$, $\lambda=K=1$, and $g_i=f_i$. A zero-dimensional stage terminates as described at the end of this section.

Choose $\gamma>0$ sufficiently small in terms of the fixed geometry, with $2\gamma<1/10$. Apply the tangent-plane cuts of External input~\ref{alginput}(2), and let $\Sigma=\bigcup_a\Sigma_a$. At dimension $n$ take $\Sigma=\varnothing$. Fix $A=100$.

If a fixed fraction of $M$ lies in $\Nbh_{Ab}(\Sigma)$, one of the finitely many $\Sigma_a$ carries a fixed fraction. Restrict $\Omega$ to that neighborhood, replace $Z$ by $\Sigma_a$, and enlarge $\lambda$ by $A$. This is a new stage of dimension $s-1$.

Otherwise retain the part outside $\Nbh_{Ab}(\Sigma)$. Choose a finite $\gamma$-net of $s$-planes. Assign each component of $Z\setminus\Sigma$ to a plane within $\gamma$ of one of its tangent planes. All tangent planes on that component are then within $2\gamma$ of the assigned plane. For each plane take the points admitting a witness in an assigned component. These measurable sets cover the retained part of $\Omega$. Assign a point belonging to several sets to the first one in a fixed ordering. One plane $V$ then retains a fixed fraction of $M$. Restrict $\Omega$ to that portion. Each $x\in\Omega$ has a witness $z\in Z$ in an assigned component with
\[
 |x-z|<b,\qquad
 \dist(z,\Sigma)>(A-1)b,\qquad
 \gap(T_zZ,V)<2\gamma.
\]
Only the finite number of planes is used; the projection need not be globally one-to-one.

Split each selected family according to
\[
 \ang(v(T),V)\le4\gamma
 \quad\text{or}\quad
 \ang(v(T),V)>4\gamma,
\]
and call these alternatives near and far. By \eqref{subadd},
\begin{equation}\label{dominance}
 \int_\Omega\prod_i|E_ig_i|^p
 \le
 \sum_{\boldsymbol\sigma\in\{\mathrm{near},\mathrm{far}\}^k}
 \int_\Omega\prod_i|E_ig_i^{\sigma_i}|^p.
\end{equation}
Choose a largest summand and replace $g_i$ by the selected sums, multiplying $K$ by at most $2^k$.

We now reset the density:
\begin{equation}\label{density}
 d\mu(x)=\one_\Omega(x)\prod_{i=1}^k|E_ig_i(x)|^p\dd x.
\end{equation}
The measure to be partitioned is this new measure. Inequality \eqref{dominance} compares total integrals and does not give a comparison on each future cell.

Let $t$ denote the number of selected far families. If the mass is nonzero, every selected family is nonempty. Choose the initial frequency patches sufficiently small that their normal variation is much less than $\gamma$. Then all normals of a nonempty near family are within $C\gamma$ of $V$. Transversality then forces
\begin{equation}\label{stcondition}
 s+t\ge k.
\end{equation}
Indeed, otherwise more than $s$ of the $k$ independent normals would lie within $C\gamma$ of an $s$-plane, contradicting \eqref{trans}.

Let $A_p(L)$ denote the best local constant at radius $L$, allowing arbitrary physical centres and the support margins described in Section~\ref{inductionsection}.

\begin{proposition}[One partition]\label{onepartition}
At a stage with $s\ge1$, put
\begin{equation}\label{cellpower}
 e(s,t,p)=s-\frac p2\bigl(k(s-1)+t\bigr).
\end{equation}
If $e(s,t,p)<0$, then for every prescribed $c>0$ a new degree $D$, independent of $R$, can be chosen so that either
\[
 M_0\le c A_p(R/2)^p\Pf^p+\Rap(R)\Pf^p,
\]
or a new stage satisfying \eqref{stageinvariant} has dimension $s-1$. Its constants depend only on the old bounds, the fixed geometry, $p$, and $c$.
\end{proposition}
\begin{proof}
Push \eqref{density} forward under $\pi_V$. This is a finite absolutely continuous measure on $V$, by Fubini. Apply External input~\ref{alginput}(1) and lift the resulting polynomial to $P=P_V\circ\pi_V$. Its sign regions satisfy
\begin{equation}\label{equalmass}
 \#\{O_j\}\asymp D^s,\qquad \mu(O_j)\asymp D^{-s}M.
\end{equation}
Choose the small generic perturbations so that the intersection of $Z$ with each factor's zero set is a smooth complete intersection.

Write $\mathcal W=\Nbh_b(Z(P))$ and $O'_j=O_j\setminus\mathcal W$. If $\mu(\mathcal W)$ is at least a sufficiently small fixed fraction of $M$, apply Lemma~\ref{lifting} with parameter $\gamma_0=2\gamma$. Its distance hypothesis follows from $A=100$. The relevant wall is within $5b$ of the union of the intersections with the partition factors. One intersection carries a fixed fraction. Replacing $Z$ by that intersection and $b$ by $5b$ gives the dimension drop.

Otherwise a fixed positive fraction of the $D^s$ cells satisfy
\[
 \mu(O'_j)\gtrsim D^{-s}M.
\]
For each such cell, and for all the other cells as well, retain the packets meeting $\Omega\cap O'_j$ and call their sums $g_{i,j}$. The subset-energy estimate and Lemma~\ref{cellcount} give
\begin{equation}\label{cellsums}
 \sum_j\|g_{i,j}\|_2^2\lesssim
 \begin{cases}
 D\|f_i\|_2^2,&i\text{ near},\\
 C_{L,\gamma}\|f_i\|_2^2,&i\text{ far}.
 \end{cases}
\end{equation}
Markov's inequality, with a sufficiently large fixed multiple of the averages, produces a substantial-mass cell for which all $k$ bounds hold:
\begin{equation}\label{cellindividual}
 \|g_{i,j}\|_2^2\lesssim
 \begin{cases}
 D^{1-s}\|f_i\|_2^2,&i\text{ near},\\
 C_{L,\gamma}D^{-s}\|f_i\|_2^2,&i\text{ far}.
 \end{cases}
\end{equation}
The constants here depend on the old stage but not on the new degree.

On $\Omega\cap O'_j$, omitted packets contribute only a rapidly decreasing error. Cover $B_R$ by a fixed number of balls of radius $R/2$ and apply the definition of $A_p(R/2)$ to the selected inputs. It follows that
\[
 \mu(O'_j)\lesssim
 A_p(R/2)^p D^{-\frac p2[k(s-1)+t]}\Pf^p
 +\Rap(R)\Pf^p.
\]
The substantial-mass lower bound and \eqref{stageinvariant} imply
\begin{equation}\label{cellestimate}
 M_0\le C_{\mathrm{old}}D^{e(s,t,p)}
 A_p(R/2)^p\Pf^p+\Rap(R)\Pf^p.
\end{equation}
Choose $D$ large enough that $C_{\mathrm{old}}D^{e(s,t,p)}\le c$. This is possible because the exponent is negative and $C_{\mathrm{old}}$ is independent of $D$.
\end{proof}

The exponent calculation can be made without a dimension-by-dimension list. At $p=\pc$,
\begin{align}
 e(k+2,0,\pc)&=-\frac1k,\label{ambientpower}\\
 e(k+1,0,\pc)&=0.\label{criticalcellpower}
\end{align}
For $1\le s\le k$, a nonterminal case has $s+t\ge k+1$, so $t\ge k-s+1$. Since $e$ decreases with $t$,
\begin{equation}\label{lowercellpower}
 e(s,t,\pc)\le e(s,k-s+1,\pc)
 =\frac{s-k-1}{k^2}<0.
\end{equation}
At $s=k+1$, positive $t$ only improves \eqref{criticalcellpower}; at $s=k+2=n$ all families are near. Thus all required cellular powers are strictly negative when $p>\pc$. In view of \eqref{stcondition}, the remaining positive-dimensional configurations are exactly
\begin{equation}\label{terminalcases}
 1\le s\le k,\qquad t=k-s.
\end{equation}

Every continuation reduces the dimension. There are at most $n$ dimension drops. At each new stage first bound the old degree, the neighborhood multiplier $\lambda$, and the mass constant $K$, and then choose the next degree. This finite hierarchy is independent of $R$ and of the coefficients of the polynomials.

If $s=0$, Bezout gives a bounded number of points. Cover $\Omega$ by the corresponding balls of radius $Cb$. The subset-energy bound yields
\begin{equation}\label{zerodimension}
 M_0^{1/p}\lesssim A_p(Cb)\Pf+\Rap(R)\Pf.
\end{equation}

\section{The terminal configurations}\label{terminalsection}

Fix one of \eqref{terminalcases}, so $s+t=k$. The functions, set $\Omega$, and algebraic variety are those selected in the descent. We continue to use the original $R$-scale packets.

Cover $B_R$ by balls $B_\ell=B(x_\ell,\rho)$ with centres on a comparable fixed lattice. Their overlap is bounded. In each ball retain the packets meeting $\Omega\cap B_\ell$. Put $\alpha_b=b/\rho$.

A retained packet is called tangent in $B_\ell$ if
\begin{equation}\label{tangentdefinition}
 \ang(v(T),T_zZ)\le\alpha_b
 \quad\text{for every }z\in Z\cap2B_\ell
 \text{ with }\dist(z,T)\le C_0b,
\end{equation}
where $C_0$ is a sufficiently large fixed constant. Otherwise it is nontangent. This is an angular condition; containment of a whole tube segment in a neighborhood of $Z$ is not assumed.

Write the two sums as $g_{i,\ell,\mathrm{tan}}$ and $g_{i,\ell,\mathrm{nt}}$. Their extensions add to $E_ig_i$ on $\Omega\cap B_\ell$ up to a rapidly decreasing error. Define
\begin{align}
 F_{\mathrm{nt}}(x)
 &=\sum_\ell\one_{\Omega\cap B_\ell}(x)
          \prod_{i=1}^k|E_ig_{i,\ell,\mathrm{nt}}(x)|,\label{Fnt}\\
 F_{\mathrm{tan}}(x)
 &=\sum_\ell\one_{\Omega\cap B_\ell}(x)
   \sum_{\substack{\boldsymbol\sigma\in\{\mathrm{tan},\mathrm{nt}\}^k\\
                   \boldsymbol\sigma\ne(\mathrm{nt},\ldots,\mathrm{nt})}}
          \prod_{i=1}^k|E_ig_{i,\ell,\sigma_i}(x)|.\label{Ftan}
\end{align}
The current absolute product on $\Omega$ is bounded by $F_{\mathrm{nt}}+F_{\mathrm{tan}}$ and a packet error.

\begin{lemma}[Nontangent terms]\label{ntlemma}
For $p$ as in \eqref{pchoice},
\begin{equation}\label{ntbound}
 \|F_{\mathrm{nt}}\|_p\lesssim A_p(\rho)\Pf.
\end{equation}
\end{lemma}
\begin{proof}
If a fixed tube $T$ is nontangent in $B_\ell$, there is $z\in Z\cap2B_\ell$ within $C_0b$ of $T$ with angle greater than $\alpha_b$. External input~\ref{alginput}(3), applied to an enlarged infinite tube, covers all such witnesses by boundedly many balls of radius
\[
 Cb/\alpha_b=C\rho.
\]
Only boundedly many lattice balls can have their doubles meeting this covering. Hence each tube is nontangent in boundedly many $B_\ell$, with a constant depending on the old degree but not on $R$. By the subset-energy estimate,
\begin{equation}\label{ntenergy}
 \sum_\ell\|g_{i,\ell,\mathrm{nt}}\|_2^2
 \lesssim\sum_Te_{i,T}\lesssim\|f_i\|_2^2.
\end{equation}
Set $b_{i,\ell}=\|g_{i,\ell,\mathrm{nt}}\|_2$. Subadditivity and induction at radius $\rho$ give
\[
 \|F_{\mathrm{nt}}\|_p^p
 \le A_p(\rho)^p\sum_\ell\left(\prod_i b_{i,\ell}\right)^p.
\]
Since $p>\pc>2/k$, monotonicity of sequence quasi-norms and H\"older give
\[
 \left(\sum_\ell\left(\prod_i b_{i,\ell}\right)^p\right)^{1/p}
 \le\left(\sum_\ell\prod_i b_{i,\ell}^{2/k}\right)^{k/2}
 \le\prod_i\left(\sum_\ell b_{i,\ell}^2\right)^{1/2}.
\]
Use \eqref{ntenergy}.
\end{proof}

Let $I$ be the set of far families, so $|I|=t$, and fix reference normals $\bar N_a$ for $a\in I$. For a $b$-cube $Q$ meeting $\Omega\cap B_\ell$, choose a witness $z_{Q,\ell}$ within $b$ of a point in that intersection, in a component assigned to $V$. Let
\begin{equation}\label{HQ}
 V_{Q,\ell}=T_{z_{Q,\ell}}Z,\qquad
 H_{Q,\ell}
 =\left(V_{Q,\ell}+\spn\{\bar N_a:a\in I\}\right)^\perp.
\end{equation}
Since $b\ll\rho$, the witness lies in $2B_\ell$ for all sufficiently large $R$.

\begin{lemma}\label{Hlemma}
The plane $H_{Q,\ell}$ has dimension two and is uniformly close to $W$ for every relevant frequency tuple. A tangent packet from family $i$ meeting $CQ$ satisfies
\begin{equation}\label{HQlocalization}
 |\pi_{H_{Q,\ell}}N_i(\zeta)|\lesssim\alpha_b
 \quad\text{on its frequency support}.
\end{equation}
\end{lemma}
\begin{proof}
There are $k-t=s$ near families. Choose a reference normal from each. Project these $s$ normals onto $V_{Q,\ell}$ and keep the $t$ far reference normals. The original tuple changes by $O(\gamma)$ and remains quantitatively independent. The projected $s$ vectors span $V_{Q,\ell}$, so the sum in \eqref{HQ} has dimension $s+t=k$. Its orthogonal complement is a two-plane close to the common normal complement. The same holds throughout the small fixed frequency patches.

If a tangent packet meets $CQ$, then $z_{Q,\ell}$ is at distance $O(b)$ from that packet. Choose $C_0$ large enough in \eqref{tangentdefinition}. Since $H_{Q,\ell}\subset V_{Q,\ell}^\perp$,
\[
 |\pi_{H_{Q,\ell}}v(T)|
 \le\ang(v(T),V_{Q,\ell})\le\alpha_b.
\]
The normal varies by $O(R^{-1/2})$ over the packet cap, and $R^{-1/2}\ll\alpha_b$. This gives \eqref{HQlocalization}.
\end{proof}

\begin{proposition}[Tangential terms]\label{tanprop}
For every $\eta>0$,
\begin{align}
 \|F_{\mathrm{tan}}\|_{\qg}
 &\lesssim_\eta R^{-1/2+2\tau+\eta}\Pf,\label{tangeneric}\\
 \|F_{\mathrm{tan}}\|_{2/k}
 &\lesssim R^{k/2}\Pf,\label{tanbase}\\
 \|F_{\mathrm{tan}}\|_\infty&\lesssim\Pf.\label{tanLinf}
\end{align}
Consequently, with
\begin{equation}\label{sigmadef}
 \sigma=\frac{k}{k+1}(2\tau+\eta),
\end{equation}
one has
\begin{equation}\label{tanLq}
 \|F_{\mathrm{tan}}\|_{\pc}\lesssim_\eta R^\sigma\Pf,
 \qquad
 \|F_{\mathrm{tan}}\|_p\lesssim_\eta R^\sigma\Pf
 \quad(\pc<p<\qg).
\end{equation}
\end{proposition}
\begin{proof}
Fix a $b$-cube $Q$, a ball $B_\ell$ meeting it in $\Omega$, and one of the $2^k-1$ products in \eqref{Ftan}. In each factor retain only packets meeting $CQ$. The omitted packets are rapidly decreasing on $Q$. Denote the resulting inputs by $h_i$. Their norms satisfy
\[
 \|h_i\|_2\lesssim a_i(Q).
\]
At least one factor is tangent, so Lemma~\ref{Hlemma} and Corollary~\ref{curvedlocal} apply with $\mu=C\alpha_b$. In integral form,
\begin{equation}\label{Qestimate}
 \int_Q\prod_{i=1}^k|E_ih_i|^{\qg}
 \lesssim_\eta b^{\eta\qg}\alpha_b^{\qg}
                  \prod_i a_i(Q)^{\qg}
       +\Rap(R)\Pf^{\qg}.
\end{equation}
The restriction theorem is applied on a ball containing $Q$; its radius differs from $b$ by a dimension-dependent constant.

A $b$-cube meets only boundedly many $\rho$-balls. Sum over cubes, balls, and product choices using \eqref{subadd}, since $\qg\le1$. Apply Lemma~\ref{energysum} with loss $\eta\qg$. Up to a rapidly decreasing error,
\[
 \|F_{\mathrm{tan}}\|_{\qg}^{\qg}
 \lesssim_\eta
 b^{\eta\qg}\alpha_b^{\qg}(R/b)^{\eta\qg}\Pf^{\qg}
 =R^{\eta\qg}\alpha_b^{\qg}\Pf^{\qg}.
\]
Taking the $1/\qg$ power and using
$\alpha_b=\lambda R^{-1/2+2\tau}$ proves \eqref{tangeneric}. Rapidly decreasing errors are absorbed by choosing their decay order sufficiently large.

For \eqref{tanbase}, define
\begin{equation}\label{Gdef}
 G_i(x)^2=
 \sum_{\ell,\sigma_i\in\{\mathrm{tan},\mathrm{nt}\}}
 \one_{B_\ell}(x)|E_ig_{i,\ell,\sigma_i}(x)|^2.
\end{equation}
Bounded overlap and the finite number of product choices give
$F_{\mathrm{tan}}\lesssim\prod_iG_i$. Local Plancherel gives
\[
 \|G_i\|_2^2\lesssim
 \rho\sum_{\ell,\sigma_i}\|g_{i,\ell,\sigma_i}\|_2^2.
\]
A tube of length $CR$ and radius $Cw\ll\rho$ is retained in at most $CR/\rho$ lattice balls. In each retained ball its tangent or nontangent classification is unique. Therefore
\begin{equation}\label{Genergy}
 \|G_i\|_2^2
 \lesssim\rho\sum_Te_{i,T}
       \#\{\ell:T\text{ is retained in }B_\ell\}
 \lesssim R\|f_i\|_2^2.
\end{equation}
H\"older with $k$ factors now gives
\[
 \|F_{\mathrm{tan}}\|_{2/k}
 \lesssim\left(\int\prod_iG_i^{2/k}\right)^{k/2}
 \le\prod_i\|G_i\|_2
 \lesssim R^{k/2}\Pf.
\]
Thus \eqref{tangeneric} and \eqref{tanbase} estimate the same positive function, including its ball-dependent packet selections. One cannot replace this step by an estimate for the unselected original product.

Every selected sum has $L^2$ norm at most $C\|f_i\|_2$. The $L^\infty$ part of Lemma~\ref{plancherel} and bounded overlap prove \eqref{tanLinf}.

Set $\theta=k/(k+1)$. The identities
\begin{equation}\label{interpolationidentity}
 \frac1{\pc}=\frac\theta{\qg}+\frac{1-\theta}{2/k},
 \qquad -\frac\theta2+(1-\theta)\frac k2=0
\end{equation}
imply, by H\"older interpolation,
\begin{align}
 \|F_{\mathrm{tan}}\|_{\pc}
 &\le\|F_{\mathrm{tan}}\|_{\qg}^{\theta}
      \|F_{\mathrm{tan}}\|_{2/k}^{1-\theta}\notag\\
 &\lesssim_\eta
 R^{\theta(-1/2+2\tau+\eta)+(1-\theta)k/2}\Pf
 =R^\sigma\Pf.\label{cancellation}
\end{align}
For $p>\pc$, use
\[
 \int F_{\mathrm{tan}}^p
 \le\|F_{\mathrm{tan}}\|_\infty^{p-\pc}
      \int F_{\mathrm{tan}}^{\pc}.
\]
This gives the stronger loss $R^{\sigma\pc/p}$, which implies \eqref{tanLq}.
\end{proof}

\section{Induction on the radius}\label{inductionsection}

We first specify the support margins. Work on fixed extended coordinate patches $U_i$. Choose $m>0$ so that the original compact supports have distance at least $m$ from their boundaries. For sufficiently large $L$, define $A_p(L)$ using inputs whose supports have distance at least
\begin{equation}\label{margin}
 m-L^{-1/4}
\end{equation}
from the boundaries of $U_i$. Allow every physical centre of the ball. A translation only multiplies the input by a phase of modulus one, preserving its norm and support.

At scale $R$, packet formation enlarges supports by at most $CR^{-1/2}$. Each child radius has the form
\[
 L=R/2,\qquad L=R^{1-\tau},\qquad
 \text{or}\qquad L=C_*R^{1/2+\tau}.
\]
For fixed $C_*$ and sufficiently large $R$,
\[
 CR^{-1/2}\le L^{-1/4}-R^{-1/4}.
\]
Thus all selected inputs are admissible at the child radius. There is only one packet decomposition during the dimension descent, so frequency enlargement is not repeated there. Bounded radii are handled by the $L^\infty$ estimate.

\begin{proposition}\label{recurrenceprop}
For fixed $p$ as in \eqref{pchoice} and sufficiently small fixed $\tau,\eta>0$, there are constants independent of $R$ such that
\begin{align}
 A_p(R)^p\le{}&
 \tfrac14 A_p(R/2)^p+C_*A_p(R^{1-\tau})^p\notag\\
 &+C_*A_p(C_*R^{1/2+\tau})^p
 +C_\eta R^{\sigma p}+\Rap(R),\qquad
 \sigma=\frac{k}{k+1}(2\tau+\eta).\label{recurrence}
\end{align}
\end{proposition}
\begin{proof}
Begin with arbitrary admissible inputs and their $R$-scale packets. Run the descent of Section~\ref{descentsection}. At a cellular termination choose the new degree in Proposition~\ref{onepartition} so that the accumulated coefficient is at most $1/4$.

If dimension zero is reached, apply \eqref{zerodimension}. If a terminal case in \eqref{terminalcases} is reached, use \eqref{Fnt}--\eqref{Ftan}, subadditivity, Lemma~\ref{ntlemma}, and Proposition~\ref{tanprop}, including the accumulated constant $K$ from \eqref{stageinvariant}.

There are at most $n$ dimension drops. Degrees, neighborhood multipliers, and mass-selection constants can be bounded successively before choosing each new degree. A single constant $C_*$ therefore covers all possible terminations. In a zero-dimensional alternative, enlarge its covering balls directly to the common radius $C_*R^{1/2+\tau}$. The margin calculation above makes the selected inputs admissible at that radius; no monotonicity of $A_p$ is used. The resulting estimates are alternatives, and their sum is also an upper bound. Divide by $\Pf^p$ and take the supremum over inputs and centres.
\end{proof}

\begin{proof}[Proof of the local estimate]
Fix $\pc<p<\qg$ and a desired loss $\varepsilon>0$. Choose $\tau,\eta>0$ with $\sigma<\varepsilon/2$. All geometric constants, packet parameters, and partition degrees are now fixed independently of $R$.

We prove
\begin{equation}\label{pestimate}
 A_p(R)\le C_{\varepsilon,p}R^\varepsilon.
\end{equation}
Inserting the inductive bounds into \eqref{recurrence}, the first term contributes at most $\frac14 2^{-\varepsilon p}$ of $C_{\varepsilon,p}^pR^{\varepsilon p}$. Relative to the same quantity, the second and third terms contribute at most
\[
 C_*R^{-\varepsilon p\tau},
 \qquad
 C_*^{\,1+\varepsilon p}R^{-\varepsilon p(1/2-\tau)}.
\]
Both are below $1/8$ above a fixed threshold. Since $\sigma<\varepsilon/2$, the forcing term contributes at most $1/4$ after increasing the base threshold and constant. The rapidly decreasing error is absorbed as well. This proves \eqref{pestimate} by induction on radius. For sufficiently large $R$, all three child radii are at most $R/2$. The local $L^\infty$ bound gives a finite base constant, and induction over successive dyadic intervals of radii is therefore legitimate, even though $A_p$ need not be monotone.

Put $q=\pc$. Given $\varepsilon>0$, choose $q<p<\qg$ so close to $q$ that
\[
 n\left(\frac1q-\frac1p\right)<\frac{\varepsilon}{2}.
\]
Use \eqref{pestimate} with loss $\varepsilon/2$ and finite-volume H\"older:
\[
 \left\|\prod_iE_if_i\right\|_{L^q(B_R)}
 \le |B_R|^{1/q-1/p}
       \left\|\prod_iE_if_i\right\|_{L^p(B_R)}
 \lesssim_\varepsilon R^\varepsilon\Pf.
\]
Summing the initial fixed frequency subdivision by \eqref{subadd} gives \eqref{localmain}.
\end{proof}

\begin{proof}[Deduction of the global strict range]
Condition \eqref{curv} implies that at least one principal curvature is nonzero at each point. For every smooth compactly supported density,
\begin{equation}\label{decay}
 |E_i\psi(x)|\lesssim_\psi(1+|x|)^{-1/2}.
\end{equation}
To see this, first work on a small patch. If $x/|x|$ is separated from the normal directions, a first derivative of the phase is bounded below and integration by parts applies. Close to a normal direction, choose a tangent coordinate with nonzero second fundamental form. Its second phase derivative has size comparable to $|x|$. One-dimensional van der Corput in that coordinate, followed by integration in the remaining coordinates, gives \eqref{decay}. A finite partition treats every smooth density.

External input~\ref{epsiloninput} applies with $u=\pc>2/k$ and arbitrary local losses. It yields \eqref{globalmain} for every $p>\pc$. It yields no lossless global endpoint estimate.
\end{proof}

\section{The next range of multilinearity}\label{obstructionsection}

There is a useful numerical distinction between the present case and lower levels of multilinearity. In this section only, allow general integers $3\le k\le n-2$ and retain the expected exponent \eqref{generalexponent}. The same cellular calculation at $(s,t)=(k+1,0)$ gives
\begin{equation}\label{obstruction}
 e(k+1,0,p(n,k))
 =k+1-\frac{k^2}{2}\frac{2(n+k)}{k(n+k-2)}
 =\frac{n-k-2}{n+k-2}.
\end{equation}
It is zero for $k=n-2$ and positive for $k\le n-3$. The positive value obstructs contraction in this particular cellular step at the expected exponent. It is not a counterexample to the expected restriction inequality.

For instance, $(n,k)=(6,3)$ gives $p(6,3)=6/7$ and
\[
 e(4,0,6/7)=\frac17.
\]
The four-linear argument in $\R^6$ therefore does not by itself settle the sharp trilinear estimate there. An additional estimate or a different organization of the descent would be needed in that case.

\section{Acknowledgements}
AI tools were used in all aspects of preparation of this article, including discovery, verification, language and editing checks etc.

We are aware of the manuscript \cite{mt}, which claims to solve the problem for all values of $k$ via application of  a general   abstract theorem. While discovery of such an abstract principle with applications to restriction theory would  have been great, a  check with AI produces counterexamples to the general abstract theorem of that  article.

\end{document}